\documentclass[12pt,oneside,mathscr]{amsart}
\usepackage{amscd,euscript,amssymb}

\usepackage[all]{xy}
\newtheorem{thm}{Theorem}[section]

\newtheorem{lemma}[thm]{Lemma}
\newtheorem{prop}[thm]{Proposition}

\theoremstyle{definition}
\newtheorem{dfn}[thm]{Definition}

\newtheorem{example}[thm]{Example}

\theoremstyle{remark}
\newtheorem{remark}[thm]{Remark}

\newcommand{\thmref}[1]{Theorem~\ref{#1}}
\newcommand{\lemref}[1]{Lemma~\ref{#1}}
\newcommand{\propref}[1]{Proposition~\ref{#1}}

\newcommand{\PP}{{\mathbb P}}
\newcommand{\OO}{{\mathscr O}}
\let\cross=\times
\let\tensor=\otimes

\newcommand{\Hilb}{\operatorname{Hilb}}

\newcommand{\Pic}{\operatorname{Pic}} 
\newcommand{\Ext}{\operatorname{Ext}}

\newcommand{\Hom}{\operatorname{Hom}}

\newcommand{\ch}{\operatorname{ch}}
\newcommand{\rk}{\operatorname{rk}}

\newcommand{\M}{{\mathcal M}}

\newcommand{\EE}{\mathbb{E}}

\newcommand{\Ltensor}{\mathbin{\buildrel{\mathbf L}\over{\tensor}}}

\newcommand{\I}{{\mathscr I}}

\newcommand{\compo}{\raise2pt\hbox{$\scriptscriptstyle\circ$}}
\renewcommand{\leq}{\leqslant}
\renewcommand{\geq}{\geqslant}

\newcommand{\lra}{\longrightarrow}

\newcommand{\R}{\mathbf{R}}
\newcommand{\E}{\mathbb{E}}

\newcommand{\lhom}{\operatorname{{\mathcal H}\mathit{om}}}

\title[]{Rank Two Fourier-Mukai Transforms for K3 Surfaces}
\author[]{Antony Maciocia}
\address{School of Mathematics\\
The University of Edinburgh\\
Peter Guthrie Tait Road\\ Edinburgh, EH9 3FD, UK.}
\email{A.Maciocia@ed.ac.uk}
\thanks{}
\subjclass{Primary: 14F05. Secondary: 14C05, 14F17, 14J28}
\keywords{Sheaf moduli spaces, Simple bundles, K3 surfaces, Hilbert schemes,
Fourier-Mukai transforms}
\begin{document}

\begin{abstract}
We study rank two locally-free Fourier-Mukai transforms on K3 surfaces and show that they come in two
distinct types according to whether the determinant of a suitable twist of the kernel is positive or not.
We show that a necessary and sufficient condition on the existence of
Fourier-Mukai transforms of rank 2 between the derived categories of K3 surfaces $X$ and $Y$
with negative twisted determinant is that $Y$ is isomorphic to $X$ and there must exist a
line bundle with no cohomology.
We use these results to prove that all reflexive K3 surfaces (including the 
degenerate ones) admit Fourier-Mukai transforms.
\end{abstract}
\maketitle

\section*{Introduction}
A Fourier-Mukai transform between two smooth projective varieties $X$ and $Y$
is an equivalence of derived categories $\Phi:D^b(X)\to D^b(Y)$ given by $\Phi(E)=\R
q_*(p^*E\Ltensor P)$, where $P$ is an object of $D^b(X\cross Y)$. The inverse
transform is given by $\hat\Phi(F)=\R p_*(q^*F\Ltensor\R\lhom(P,\OO_{X\cross
Y}))$ up to a shift of complexes. Such transforms were first studied by Mukai
in the case where $X$ is an abelian variety and $Y$ is the dual abelian
variety (see \cite{Muk1}); the Poincar\'e bundle provides the object $P$.
An interesting set of these occur when $X$ is a K3
surface and $Y$ is a two dimensional smooth compact fine moduli space of sheaves on $X$. In these
cases, $P=\E$ is a universal sheaf for this moduli space. In fact, one
of the characteristic features of such a Fourier-Mukai transform is that
$\E$ is {\em bi-universal}: in other words, $X$ is a moduli space of
sheaves on $Y$ with universal sheaf $\E$. See \cite{BBH} and \cite{BBH2} for the first examples
of such transforms and there are further developments in \cite{BBHBook}. The aim of this article
is to complete the description of these Fourier-Mukai transforms whose kernels are rank 2 vector
bundles on a K3 surface. We will show that these
cannot exist when the Picard rank is 1 and are essentially the only additional examples of rank 2
Fourier-Mukai transforms when the Picard rank is bigger than 1. There is a detailed description
of the Picard rank 1 case in \cite{HLOY1} and there is further information in \cite{Yosh99}. We
reproduce these in \propref{p:pic1} with a quick proof in the spirit of this paper.

Such transforms are important because, via base change and their functorality,
they preserve moduli functors. So if a moduli space $\M$ of simple torsion-free
sheaves has transforms given by torsion-free sheaves then the image $\hat\M$
of $\M$ under the transform is also a moduli of simple sheaves and the map
$\M\to\hat\M$ is biholomorphic and a hyperK\"ahler isometry with respect to
the standard $L^2$ metrics on the moduli spaces. In particular, if the rank of
the sheaves in $\hat\M$ is one then we obtain an isometry between $\M$ and a
Hilbert scheme of points on the variety. An example of such results can be found in \cite{BM}.

In this paper we shall be concerned with the case of a rank two
vector bundle $\E$ over a product of $X\cross Y$ where $X$ is a K3 surface. We denote the
restrictions to $X\times\{y\}$ by $E_y$. Then we
can prove the following:
\begin{thm}
If $X$ is a projective K3 surface and $Y$ is a scheme then
there is a Fourier-Mukai transform between $X$ and $Y$ whose kernel is a rank 2 locally-free
sheaf with $c_1(M^{-2}\otimes E_y)^2\leq 0$ for a certain line bundle $M$ (we shall call it
\emph{generically special} with respect to $E_y$ and define this in Definition \ref{d:spec} below)  if and only if
 $Y\cong X$ and, under this isomorphism,
there exist line bundles $A$, $B$, $C$
and $D$ such that $A\tensor \phi^*B\cong C\tensor \phi^*D$ for an isomorphism $\phi:X\to Y$ and
$H^i(A\tensor C^*)=0$ for all $i$, and a non-trivial extension
\begin{equation}
0\lra A\boxtimes B\lra\E\lra(C\boxtimes D)(1\cross\phi)^*\I_\Delta\lra0,\label{e:genE}
\end{equation}
where $\I_\Delta$ is the ideal
sheaf of the diagonal $\Delta$ in $X\cross X$.

In particular, there must exist a line bundle on $X$ with no cohomology.
\end{thm}
On the other hand if the Picard rank is one then such transforms cannot exist.

We proceed by first assuming that $Y\cong X$ and that $\E$ takes the form
(\ref{e:genE}) for some line bundles $A$, $B$, $C$ and $D$, and then show
that $\E$ gives rise to a Fourier-Mukai transform if and only if $A\tensor
B=C\tensor D$ and $H^i(A\tensor C^*)=0$. This is done in \thmref{t:suffic}. We
then go on to remove this ansatz. In the case when $c_1(M^{-2}\otimes E_y)^2\geq0$ we show that the
moduli space $Y$ is stratified by rational curves corresponding to the maximal slope of sub-line
bundles of $\E_y$.

We shall also see that the known (rank 2) examples of Fourier-Mukai transforms
satisfy the conditions of the theorem and so we obtain new proofs that such
transforms exist.
These occur for so called {\em reflexive} K3 surfaces introduced in \cite{BBH} which we define in \S4
(the constraint on the surface is simply the existence of certain divisors). The original
examples required also that a certain non-degeneracy condition be satisfied. Bruzzo, Bartocci
and Hern\'andez-Ruip\'erez describe the moduli of such non-degenerate K3 surfaces in \cite{BBH3}.
We shall use our theorem
to extend the result in \cite{BBH} to the degenerate case as well. It is also
easy to extend the results of \cite{BM} giving explicit biholomorphic (and
indeed, hyper-K\"ahler)
isomorphisms between certain moduli spaces of simple sheaves and Hilbert
schemes of zero dimensional subschemes. We extend these results
to the more general setting where our K3 surface is only assumed to admit
some line bundle $L$ with no cohomology.

Throughout the paper we use $AB$ as a shorthand for the tensor $A\otimes B$ of sheaves.
\section*{Acknowledgements}
This article is dedicated to my friend and colleague Ugo Bruzzo on his 60th birthday. I am also
grateful to Kota Yoshioka for pointing out an error in an earlier version of this paper.

\section{Fourier-Mukai transforms for K3 surfaces}
Let $X$ be a (minimal) K3 surface. In this section, we shall study Fourier-Mukai transforms
$\Phi_\EE:D^b(X)\to D^b(Y)$ with
$Y=X$ and $\EE$ given by an extension of the form
\[0\lra A\boxtimes B\to\EE\lra(C\boxtimes D)\I_\Delta\lra 0,\]
where $A$, $B$, $C$ and $D$ are lines bundles on $X$ and $\I_\Delta$ is the ideal sheaf of the diagonal
$\Delta\subset X\times X$.
We let $c_1(A)=a$,
$c_1(B)=b$, $c_1(C)=c$ and $c_1(D)=d$. Let $E_x$ denote the restriction of $E$
to $X\cross\{x\}$ and let ${}_xE$ denote the restriction to $\{x\}\cross Y$.
We shall assume that $\E$ is locally-free.

Under this assumption on $\E$ we can write down a general formula for the Chern
character of a transform.
\begin{prop}\label{p:necconds}
If $F$ is a sheaf on $X$ with Chern character $(r,f,t)$ then the Chern
character of $\Phi(F)$ is given by
\begin{align*}
\ch_0&=\textstyle\frac12r(a^2+c^2+6)+f\cdot(a+c)+2t\\
\ch_1&=\textstyle\frac12r\bigl[(a^2+4)b+(c^2+2)d-2c\bigr]+(f\cdot a)b+
(f\cdot c)d-f+t(b+d)\\
\ch_2&=\textstyle\frac14r(a^2b^2+4b^2+c^2d^2-2c^2+2d^2-4c\cdot d)+{}\\
&\qquad{}+\textstyle\frac12 f\cdot((d^2-2)c+b^2a-2d)+
\textstyle\frac12t(b^2+d^2-2)
\end{align*}
\end{prop}
\begin{proof}
This follows by an easy computation from the formula
$$\ch(\Phi F)=\chi(FA)\ch(B)+\chi(FC)\ch(D)-\ch(FCD)$$
which is obtained by applying $\R q_*\compo p^*F\tensor{-}$ to Sequence
(\ref{e:genE}). 
\end{proof}
If $\Phi$ is to give rise to a Fourier Inversion Theorem then we require
that the restrictions to the factors parametrize complete two dimensional
simple moduli spaces. This immediately imposes a necessary condition on the
line bundles $A$, $B$, $C$ and $D$. Then $\chi(E_y,E_y)=(a-c)^2+4=2\chi(AC^*)$. But the
dimension 
of the moduli of simple sheaves is $2-\chi(E,E)$ and so we must have
$\chi(AC^*)=0$ and similarly $\chi(BD^*)=0$. If we further suppose that
$\det\Phi(\OO)=\OO$ which we can arrange by twisting $\E$ by a line bundle
pulled back from the second factor then we obtain a constraint on $A$, $B$, $C$ and $D$
which can be expressed by
\begin{equation}\label{e:constd}
(\chi(C)-1)d=c-\chi(A)b.\end{equation}

The following can be easily proved from (\ref{e:genE}) and is our first
indication of the role of line bundles with no cohomology.
\begin{prop}\label{p:AOBO}
Let $\E$ be given as in (\ref{e:genE}) and suppose that $\Phi\OO=\R
q_*\mathbb{E}=\OO$. Then $A=\OO=B$, $C=D^*$ and $h^i(C)=0$ for $i=0,1,2$.
\end{prop}
We shall see below that the conditions on $A$, $B$, $C$ and $D$ in this
proposition are also sufficient for the existence of a Fourier-Mukai transform
given by $\E$.

\section{Strongly Simple Families}
Following a celebrated result of Bondal and Orlov \cite{BO}, we make the following definition:
\begin{dfn}\label{d:strsmp}
Let $\E\to X\cross S$ be a family of sheaves over any projective variety $X$ parametrized by a
smooth variety $S$. We say that $\E$ or $\{\E_s\}_{s\in S}$ is {\em strongly
simple} if, for all distinct geometric points $s,t\in S$ and $i$,
$\dim\Ext^i(\E_s,\E_t)=0$ and $E_s$ is simple. Such a family is {\it complete}
if it is a component of the moduli scheme of simple sheaves (when such a scheme
exists). 
\end{dfn}
Bondal and Orlov's result can then be rephrased as ``strongly simple families give rise to fully
faithful functors''.
\begin{example}
For a K3 surface $X$ and any line bundle $L$ on $X$, the family $\{L\I_x\}_{x\in X}$ is
strongly simple and complete. 
\end{example}
\begin{remark}
For a K3 surface observe that $\chi(E,E)=0$ means that the condition
$$\dim\Hom(E_s,E_t)=\delta_{st}$$
is equivalent to strong simplicity for a 2 dimensional moduli of simple
sheaves. Mukai (\cite{Muk3}) has already shown that the simple moduli space exists and is
smooth for K3 surfaces.
\end{remark}
We can use the example above to generate strongly simple families of
torsion-free sheaves.
\begin{lemma}\label{l:strong}
Suppose that $X$ is a K3 surface.
Let $A$ be a simple torsion-free sheaf and $C\to X\cross S$ a strongly simple
family of torsion-free sheaves flat over an algebraic variety $S$ such that,
for all geometric points $s\in S$, $\Hom(A,C_s)=0$ and $\Hom(C_s,A)=0$.
Then the set of non-trivial extensions  
$$0\lra A\lra E\lra C_s\lra0$$
modulo isomorphisms is a strongly simple family over $X\cross S$.
\end{lemma}
\begin{proof}
Suppose that $E$ and $F$ are two such extensions and let $\alpha:E\to F$ be a
non-zero map:
\[\xymatrix{0\ar[r]&A\ar[r]& E\ar[r]\ar[d]_{\alpha}&C_s\ar[r]&0\\
0\ar[r]&A\ar[r]&F\ar[r]&C_r\ar[r]&0}
\]
Observe that, if $A\to E\to F$ is zero then there is an induced non-zero map
$C_s\to F$. But the image of this in $C_t$ is non-zero from the hypothesis
$\Hom(C_s,A)=0$. But by the strong-simplicity this means that $s=t$ and
$C_s\to F$ splits the $F$ sequence; a contradiction.
So we may assume that $A\to F$ is non-zero. From the hypothesis
$\Hom(A,C_s)=0$,
the composite $A\to E\to F\to C_t$ is zero and so there is a
non-zero lift $A\to A$. But $A$ is simple and so this must be a multiple of
the identity. Then a simple diagram chase gives a non-zero map $C_s\to C_t$.
But this implies that $s=t$ otherwise we must have $\alpha=0$. On the other
hand, if $\alpha\neq0$ then, since $C_s$ is simple, we must have
that $\alpha$ is an isomorphism as required.
\end{proof}
We now establish our main theorem under the assumption that $X\cong Y$ and
that $\E$ is given by an extension (\ref{e:genE}).
\begin{thm}\label{t:suffic}
Let $X$ be a projective K3 surface.
Suppose that $A$, $B$, $C$ and $D$ are line bundles over $X$. Then a
locally-free extension
$$0\lra A\boxtimes B\lra\E\lra (C\boxtimes D)\I_\Delta\lra0$$
gives rise to a Fourier-Mukai transform if and only if the two conditions
\begin{equation}\label{e:equivs}
\mbox{(i) }h^j(AC^*)=0\mbox{ for all $j$}\mbox{, and  (ii) }AB=CD\end{equation}
are satisfied.
\end{thm}
\begin{proof}
As a first step we can assume that $A=\OO=B$ by twisting $\E$ by $A^*\boxtimes
B^*$ which does not affect the condition that $\E$ gives rise to a
Fourier-Mukai transform. We must prove that the conditions (\ref{e:equivs}) are
equivalent to the statement that the restrictions to the factors form 2
dimensional complete strongly simple families of vector bundles (\cite{Br}). Since the
conditions are symmetric in $C$ and $D$ we need only consider the restrictions
to the first factor. 

\begin{lemma}
Let $C$ be some line bundle on a K3 surface $X$. Then the moduli of
non-trivial extensions of the form
$$0\lra\OO_X\lra E_x\lra C\I_x\lra0,$$
where $x$ varies over $X$, forms a complete strongly simple two dimensional
family of vector bundles if and only if $H^i(C)=0$ for all $i$.
\end{lemma}
\begin{proof}
The condition that $E_x$ be locally-free is determined by the
Cayley-Bacharach condition which tells us that $E_x$ is locally-free if and
only if $x$ is not a zero of any section of $C$. As this
must hold for all $x$ it is equivalent to $H^0(C)=0$. When $H^0(C)=0$ then we
can deduce that all such non-trivial extensions are locally-free.
But since $E_x$ is such an extension and is simple, the Semi-Continuity
Theorem implies that such extensions are generically simple. Then the
dimension of the moduli of such simple sheaves implies that
$\dim\Ext^1(C\I_x,\OO)=1$ and so $h^1(C)=0$.
From $\chi(E_x,E_x)=2\chi(C)$ we see that $\chi(C)=0$ is equivalent to the
fact that the deformation space of $E_x$
is 2 dimensional (and smooth, see \cite{Muk2}). Hence, we also have
$H^2(C)=0$. Conversely, if the cohomology vanishes then \lemref{l:strong}
implies that such extensions form a complete strongly simple family of vector
bundles. This completes the proof.
\end{proof}
This deals with condition (i) of the theorem.
To complete the proof all we we need to show is that there exists a
non-trivial extension (\ref{e:genE}) and that its restriction to the
factors is always non-trivial. To achieve this consider part of the long
exact sequence given by applying $\Ext^*(-,\OO)$ to the structure sequence of
$\Delta$ twisted by $C\boxtimes D$:
\begin{multline*}\Ext^1(C\boxtimes D,\OO)\to\Ext^1((C\boxtimes D)\I_\Delta,\OO)\to\\ \to
\Ext^2((C\boxtimes D)\OO_\Delta,\OO)\to \Ext^2(C\boxtimes D,\OO).\end{multline*}
By the K\"unneth formula we see that the first and last terms are zero since
$C$ has no cohomology. On the other hand, $\Ext^2((C\boxtimes
D)\OO_\Delta,\OO)$ is naturally isomorphic to $H^0(C^*D^*)$ using Serre
duality (twice) and a degenerate Leray spectral sequence. By the naturality of
$\Ext^*$, we have a collection of commuting diagrams indexed by $x\in X$:
\[\xymatrix{\Ext^1((C\boxtimes D)\I_\Delta,\OO)\ar[r]^(0.60){\sim}\ar[d]_{\text{restriction}}&
    H^0(C^*D^*)\ar[d]^\beta\\ 
\Ext^1(C\I_x,\OO) \ar[r]^{\sim}& \Ext^2(\OO_x,\OO)}\]

But $\Ext^2(\OO_x,\OO)\cong H^0(\OO_x)$ and the map $\beta$ is just the
restriction of sections, i.e. evaluation of sections at $x$. Then a given
extension $\E$ gives rise to non-trivial extensions $E_x$ for all $x$ if and
only if $C^*D^*$ has a nowhere vanishing section.  This is equivalent to
$C^*D^*=\OO$ as required.
\end{proof}

We can now use this theorem to state the converse to \propref{p:AOBO}
\begin{thm} \label{t:Lnocohom}
Suppose $X$ admits a line bundle $L$ with no cohomology. Then the line bundles
$A=\OO=B$, $C=L$ and $D=L^*$ give rise to a Fourier-Mukai transform. The Chern
character of $\Phi(F)$ is given by
\begin{align*}
\ch_0&=r+f\cdot\ell+2t \\
\ch_1&=-(f\cdot\ell+t)\ell-c \\
\ch_2&=-2f\cdot\ell-3t,
\end{align*}
where $\ch(F)=(r,f,t)$ and $\ell=c_1(L)$. Moreover, $\Phi(\OO)=\OO$.
\end{thm}

\begin{remark}
Note that $(C\boxtimes D)\I_\Delta$ is the kernel of a spherical twist up to shift as it is the kernel of the natural map $C\boxtimes D\to\OO_\Delta$.
\end{remark}

\section{Properties of the Transforms}

We now aim to prove our main classification theorem. The key is the following notion:
\begin{dfn}\label{d:spec}
Suppose $E$ is a locally-free sheaf of rank $2$ on a projective surface $X$. Then we say that a
line bundle $M$ is \emph{special} for $E$ if $E/M$ is torsion-free and is \emph{generically
  special} for $E$ if $E'/M$ is torsion-free for all generic deformations of $E$ (meaning that
if $\EE$ over $X\times S$ is a local deformation then $E_s/M$ is torsion-free for some non-empty
Zariski open subset of $S$). We do not necessarily require $E/M$ to be torsion-free.
\end{dfn}

\begin{thm}\label{t:main}
Suppose that $X$ is a projective K3 surface and $Y$ is some
scheme. Suppose also that $\mathbb{E}\to X\cross Y$ is a
rank 2 locally-free sheaf giving rise to a Fourier-Mukai transform between $X$ and
$Y$.  Suppose that there is a line bundle $M$ which is generically special for some $E_y$ such
that $c_1(M^{-2}E_y)\leq0$.  
Then there is a natural isomorphism $\phi:Y\to X$ and there exists a
line bundle $L$ on $X$ with $H^*(L)=0$ such that $\mathbb{E}$ takes the form
$$0\lra A\boxtimes B\lra\mathbb{E}\lra(C\boxtimes
D)\I_{(1\cross\phi)^*\Delta}\lra0,$$
where $A$, $B$, $C$ and $D$ are line bundles such that $A^*C=L=D^*B$. Moreover, if $X$ has no
rational curves then it suffices to assume $M$ is special for some $E_y$.
Conversely, any such extension gives rise to a Fourier-Mukai transform.
\end{thm}
The main ingredient is the following
proposition which, together with \thmref{t:Lnocohom}, gives us the Theorem.
\begin{prop}\label{p:main}
Suppose $\{E_s\}_{s\in S}$ is a complete 2-dimensional family of simple rank 2
locally-free sheaves over a projective K3 surface $X$  whose universal bundle $\mathbb{E}$
gives rise to a Fourier-Mukai transform $\Phi$. If $M$ is generically special for some
$E_s$ such that $c_1(M^{-2} E_s)^2\leq0$ then $H^*(L\otimes M^{-2})=0$, where
$L=\det(E_s)$ for some $s$. Moreover, there is an isomorphism $\phi:S\to X$
so that for all $s\in S$ we have an extension
\begin{equation}0\lra M\lra E_s\lra LM^*\I_{\phi(s)}\lra0.\label{eq:es}\end{equation}
\end{prop}
\begin{proof}
Since $\Pic X$ is discrete, $X$ is projective and $E_s$ are locally-free
there exists a line bundle $M$ such that for all
$s\in S_0$, a dense open subset of $S$, we have injections $M\to E_s$ with
$E_s/M$ torsion-free. Let $E_s/M=LM^*\I_{Z(s)}$, for $Z(s)\in\Hilb^n X$ for some
$n>0$. If we remove this open subset from $Y=\M(E_s)$ then we can repeat this
construction and in this way stratify $Y$. The highest stratum is given by the
maximal degree subsheaf of $E_s$.
Note that, for $s\in S_0$, we have $\Hom(E_s,M)=0$ since otherwise we could
compose with $M\to E_{s'}$ to get a non-trivial map $E_s\to E_{s'}$. By a
similar argument we also have $\Hom(LM^*\I_{Z(s)},M)=0$. This implies that
$H^0(L^*M^2)=0$. 

For simplicity we shall normalize $E_s$ for $s\in S_0$ so that $M=\OO$ and denote $c_1(L)$ by
$\ell$. Our aim is to show that 
that $|Z(s)|=1$. Notice first that the condition $\chi(E_s,E_s)=0$
implies that $c_1(L)^2=4|Z(s)|-8$ and so if $c_1^2(E)\leq0$ we have $|Z(s)|\leq1$. The case
$|Z(s)|=0$ cannot happen in every stratum and so $|Z(s)|=1$ for some $s$ and hence all $s$.
\end{proof}
It is interesting to analyze the cases when $c_1(L)^2>0$. Then
$\chi(L)=2|Z(s)|-2$. We already know
$H^2(L)=0=H^2(E_s)$. Note also that $\chi(E_s)=|Z(s)|>1$.

Fix a polarizing class $h$ on $X$ and let $|Z(s)|=z$ for $s\in S_0$.
 In what follows we let $\Phi$ denote the Fourier-Mukai transform given by
 $\mathbb{E}^*$. Recall that an object $G$ is said to satisfy $\Phi-$WIT$_i$ if $\Phi^j(G)=0$
 unless $j=i$. In this case, the fibre of $\Phi^i_\EE(G)$ at $y\in Y$ is naturally isomorphic to
 $H^i(G\Ltensor\EE_y)$. If furthermore $\Phi^i(G)$ is locally-free we say that $G$ satisfies
 $\Phi-$IT$_i$. Equivalently, $\dim H^i(G\Ltensor\EE_y)$ is constant in $y$. The following
 holds even when $z=1$.
\begin{lemma}\label{l:LIZ}
For all $s\in S_0$, $L\I_{Z(s)}$ satisfies $\Phi-$WIT$_1$. 
\end{lemma}
\begin{proof}
Note that $\Hom(L\I_Z,E_{s'})=0$ for all $s'\in S$ by the strong simplicity assumption. So
$\Phi^2(L\I_{Z(s)})=0$. If $E_s\not\cong E_{s'}$ so $s'\in S_0$ then $\Hom(E_{s'},L\I_{Z(s)})=0$
because any such map must surject by the minimality assumption on $L$. Hence
$\Phi^0(L\I_{Z(s)})=0$ as its generic section is $H^0(E_{s'}^*L\I_{Z(s)})$. 
\end{proof}
\begin{lemma}
$\OO$ satisfies $\Phi-$WIT$_2$.
\end{lemma}
\begin{proof}
Apply $\Phi$ to 
\begin{equation}0\to\OO\to E_s\to L\I_{Z(s)}\to 0\label{eq:structE}
\end{equation}
and apply the previous lemma.
\end{proof}
For now fix a polarization
class $h$. Then by Hodge Index Theorem, $\ell\cdot h=0$ implies $\ell^2=0$ which is impossible as
then $\chi(L)>0$. So $\ell\cdot h>0$. For $s\in S_0$ it follows that $E_s$ is $\mu$-stable (with
respect to any polarization).
If $E$ is in a lower stratum then such $E$ are only stable if $\mu(M)<\mu(LM^*)$. Since the
deformation of extensions of $M$ by 
$M^*L$ are unobstructed we must have that such $E$ are stable since $\Phi(E)=0$ for any stable
$E$ with Chern character $(2,\ell,z-4)$.  But now $\M(E_s)$ must be a fine moduli space which contains some
stable points and so there cannot exist semistable points. Hence $\mu(M)<\frac12\mu(L)$.
Summarizing we have the following chain of strict inequalities:
\[0<\mu(M)<\frac12\mu(L)<\mu(LM^*)<\mu(L)\leq h^2.\]
 It also follows that $z>2$ and that $H^1(L)=0$.
Consider a $s\in S$ with a short exact sequence
\begin{equation}
0\to M\to E_s\to LM^*\I_{Y}\to 0\label{eq:structM}
\end{equation}
where $c_1(M)=m$ and $\frac12\ell\cdot h\geq m\cdot h>0$. Let $|Y|=y$. Since $\chi(E_s,E_s)=0$ we have
$m\cdot(\ell-m)=z-y$. Note that $H^2(M)=0$. Since $\mu(LM^*)<\mu(L)$ we have
$\Hom(E_t,LM^*\I_Y)=0$ for generic $t\in S$ and so $\Phi^0(LM^*\I_Y)=0$. Hence,
$\Phi^1(M)=0$. So $M$ satisfies $\Phi-$WIT$_2$ and $\Phi^2(M)$ is supported on the set of $t$
such that $\Hom(M,E_t)\neq0$. But then this must be at least dimension $1$ as otherwise
$\Phi^{-1}\Phi(M)\not\cong M$. It must also be rigid as the WIT condition is open but $M$ is
rigid. It follows also that $\chi(M,E_S)=0$. Hence, $m\cdot(\ell-m)=z$ 
and, in particular, $y=0$.

Define the set
\begin{multline*}\Delta=\{M\in\Pic(X) \mid 0\leq \mu(M)<\frac12\mu(L)\\ \text{ and }\exists s\in S, 
M\hookrightarrow
  E_s\text{ s.t. }E_s/M\text{ is t.f}\}.\end{multline*}
Our choice of polarization orders $\Delta$ by slope. Define the positive integer $a$ by
$\mu(\Delta)=[0,a]$. Let $L\Delta^*$ be the set of $LM^*$ for $M\in\Delta$. Then this image of
$L\Delta^*$ under $\mu$ is $[\mu(L)-a,\mu(L)]$. 

The argument above shows that any $M\in\Delta$ which is not $\OO$ must have $\mu(M)>0$ and then
satisfies $\Phi-$WIT$_2$ with transform supported on a rigid divisor. By symmetry we can then
assume that $X$ contains rigid divisors. Hence, we have
\begin{prop}
If $X$ has no rational curves then there are no lower strata.
\end{prop}
In particular, if $M$ is special for some $E_s$ then it is special for all $E_s$ (and so
generically special for some $E_s$).
\begin{lemma}\label{l:noexist}
If $M$ is a line bundle such that $a<\mu(M)<\mu(L)-a$ then $M$ and $LM^*$ satisfy
$\Phi-$IT$_1$ and $\ell\cdot m-m^2>z$. Moreover, $L$ and $M$ are independent in the
Neron-Severi group. 
\end{lemma}
\begin{proof}
Suppose that $M$ is such a line
bundle with $a<m\cdot h<\ell\cdot h-a$. Then $\Hom(M,E_s)=0=\Hom(E_s,M)$ for all $s\in S$ and so
$\Phi^2(M)=0=\Phi^0(M)$. It also follows that $H^2(M)=0$ and $\Phi^1(M)$ is
locally-free. Then $\chi(ME_s^*)<0$ and this gives $\ell\cdot m-m^2>z$. 

Now suppose $\alpha m=\beta\ell$ for some integers $\alpha$ and $\beta$. Then $4\beta^2
z=\beta^2\ell^2+8\beta^2=\alpha^2 m^2+8\beta^2<4\beta^2(\ell\cdot m-m^2)=4\alpha\beta
m^2-4\beta^2m^2$. Then $8\beta^2<-(\alpha-2\beta)^2m^2$. So $m^2<0$. But $\ell^2\geq0$, a contradiction.
\end{proof}
\begin{prop}\label{p:pic}
If  $\ell$ is ample then $\ell$ cannot be written as a multiple of another ample divisor.
\end{prop}
\begin{proof}
Set $\ell= nh$ (by choice of $h$) and $m=h$ and then apply the previous lemma.
\end{proof}

When the Picard rank is 1 then such line bundles $M_i$ cannot exist and there are unique rank 2
Fourier-Mukai transforms. This was proved by Hosono, Lian, Oguiso and Yau in \cite{HLOY1}. For
completeness we apply sheaf 
theoretic techniques to study this case. Observe first that \propref{p:pic} implies
that $\Pic(X)=\langle L\rangle$. It also
follows that $\rho(S)=1$. Note also that $4|\ell^2$. Let $\Pic(S)=\langle \hat L\rangle$. 
By symmetry, the restriction of $\EE$ to $S$ can also be written as extensions of $\hat
L\I_{\hat Z}$ by $\OO_S$.

\begin{prop}\label{p:pic1}
Let $X$ be a polarized K3 surface with $\Pic(X)=\langle L\rangle$. Then there is a Fourier-Mukai
transform $D^b(X)\to D^b(Y)$ with locally-free kernel of rank 2 if and only if $\ell^2\equiv
4\pmod 8$ and then each restriction $E_y$ to $X$ of the kernel sits in an extension of the form
\[0\lra L^k\lra E_y\lra L^{k+1}\I_Z\lra 0\]
for some integer $k$ and 0-dimensional scheme $Z$ of length $2n+3$, where $\ell^2=4(2n+1)$, for
some $n\geq0$. Moreover the Chern character of $\Phi_{\EE}(F)$ is given by
\begin{align*}
\ch_0&=(2n+3)r+c\ell^2+2t\\
\ch_1&=\bigl((n+1)r+c(4n+1) +t\bigr)\hat\ell\\
\ch_2&=2(n^2-1)r+(n-1)c\ell^2+(2n-1)t
\end{align*}
where $\ch(F)=(r,c\ell,t)$.
\end{prop}
\begin{proof}
The Picard rank 1 case is studied in \cite{HLOY1} and it is shown that $\ch(E_s)=(r,\ell,s)$
with $r$ and $s$ coprime. When $r=2$ we have that $\ell^2/4=z-2$ must be odd and so $z$ is also
odd. Then we can write $\ell^2=4(2n+1)$ and $z=2n+3$. It is also shown that $Y$ is a K3 surface
with Picard rank 1 and $\hat\ell^2=\ell^2$ for a generator $\hat\ell$ of $\Pic(Y)$.

  Recall that $\OO$
is $\Phi-$WIT$_2$. Let $\ch(\Phi^2(\OO))=(z,c\hat\ell,\alpha)$. Twisting by line bundles on $S$
we can assume that $\ch(\Phi(\OO_x))=(2,\hat\ell,z-4)$.  Then we
can express the linear map given by $\Phi$ induced on algebraic K-Theory of $X$ in the same
basis by the matrix 
\[\Phi^K=\begin{pmatrix} z&-\ell&2\\ c\hat\ell&\phi&-\hat\ell\\\alpha&y\ell&z-4 
  \end{pmatrix}\]
for some integer $y$ and linear map $\phi:\Pic(X)\to\Pic(Y)$. Let
$\phi(\ell)=x\hat\ell$. We can determine these unknowns
as follows.
Applying the matrix to $(2,\ell,z-4)$ we should get $(0,0,1)$ as $\Phi(E_s)=\OO_s$. But
$\Phi^K(2,\ell,z-4)=0,2c\hat\ell+x\hat \ell-(z-4)\hat\ell,2\alpha+4y(z-2)+(z-4)(z-4))$. The
middle term gives $x=z-4-2c$ and the last term gives $2\alpha+4y(z-2)+(z-4)(z-4)=1$. Using 
$\chi(\OO,\OO)=2=\chi(\Phi(\OO),\Phi(\OO))$ we have 
\begin{equation}2z\alpha-4c^2(z-2)+2z^2=2.\label{eq:OO}
\end{equation}
We also have
$1=\rk(\Phi^{-1}(\Phi(\OO))=\chi((2,\hat\ell,z-4)(z,c\hat\ell,\alpha))=2\alpha+4c(z-2)+z(z-4)+4z$.
Multiplying this by $z$ and subtracting the previous equation we obtain
\[4c(z+n)(z-2)=(z-2)(1-z^2).\]
Rearrange to get $(z+2c)^2=1$. Then $z=1-2c$ or $z=-1-2c$. In either case, substituting into
(\ref{eq:OO}) we find $\alpha=2c(2+c)$ and the other equation gives $y=c+2$.

In the former case $c=-n-1$ and $x=4n+1$ and in the latter $c=-n-2$ and $x=4n+3$. These give
$|\Phi^K|=1$ and $|\Phi^K|=-1$ respectively. Only one such transform can occur as their
composition would be the zero functor (because $\Ext^*(E_y,E'_y)=0$, where $\EE$ and $\EE'$
are the two respective kernels). We will show that it must be the former. To see this observe 
that $\OO$ is 
$\Phi^{-1}-$IT$_0$ and a surjection $\OO\to \OO_W$ for any 0-dimensional subscheme $W$ gives a
non-zero map $\widehat\OO\to\oplus E_s$. In the second case
$\mu(\widehat\OO)=\dfrac{n+2}{2n+3}\ell^2>\mu(E_s)=\dfrac12\ell^2$ and so 
$\widehat\OO$ would be unstable. But $\widehat\OO$ is an exceptional vector bundle and these are
always stable when the Picard rank is 1 (see \cite[Theorem]{Z}).
\end{proof}
\begin{remark}
In \cite{HLOY1} it is also shown that $X\cong Y$ exactly when $n=0$.
\end{remark}
\begin{remark}
We can say rather more. Since $E_s$ move in a strongly simple family and $\Ext^1(E,\OO)=0$ we
have that $E_s$ 
determine a unique family of 0-dimensional subschemes $Z$ which must cut out the cohomology
jumping locus $J$, say, in $\Hilb^zX$ parametrizing sheaves of the form $L\I_Z$. Then $\dim
J=h^0(E_s)-1+2=2n+4$. This is codimension $2n+2$ in $\Hilb^zX$. The Cayley-Bacharach condition
implies that any colength 1 subscheme of $Z\in J$ is not in the cohomology jumping locus of
$\Hilb^{z-1}X$. 
\end{remark}

\section{Reflexive K3 surfaces}
\begin{dfn} Following \cite{BBH} we define a {\em reflexive}
K3 surface to be a K3 surface which admits two line bundles $H$ and $L$ such
that $H$ is ample (and regarded as a polarization of $X$) and
$$H^2=2,\qquad L^2=-12\qquad H\cdot L=0.$$
Let $h=c_1(H)$ and $\ell=c_1(L)$.
We say that the reflexive K3 surface is {\em degenerate} if $H^0(LH^2)\neq0$
and {\em non-degenerate} otherwise.
\end{dfn}
Note that $\chi(LH^2)=0$.
Since $h$ is ample and $(\ell+2h)\cdot h=4$ we have that $LH^2$ has no
cohomology on a non-degenerate reflexive K3 surface. Note also that if we set
$\hat L=L^5H^{12}$ and $\hat H=L^2H^5$ then $\hat H$ and $\hat
L$ satisfy the same conditions as $H$ and $L$. One can show that $\hat H$ is
ample if $X$ is non-degenerate using the methods of
\lemref{l:d1d2} below or indirectly using fact that $\hat H^2$ is the
determinant line bundle (see \cite{BBH2}). Conversely, if $X$ is degenerate
then that lemma
implies that $\ell+2h=d_1+d_2$, where $d_1$ and $d_2$ are effective of
self-intersection $-2$ with $d_1\cdot d_2=0$. Then $2\ell+5h=2d_1+2d_2+h$. But
$h\cdot d_1\leq3$ and so $(2\ell+5h)\cdot d_1<0$ and so $\hat H$ is not ample.

\thmref{t:suffic} immediately implies the first part of the following.
\begin{thm}[{\rm see} \cite{BBH}]\label{t:nondegFM}
If $X$ is a non-degenerate reflexive K3 surface then 
$$0\lra H^{-1}\boxtimes L^3H^7\lra\E\lra(LH\boxtimes L^2H^5)\I_\Delta\lra0$$
gives rise to a Fourier-Mukai transform $\Phi$
and the Chern character of $\Phi(F)$ is given by
\begin{align*}
\ch_0&=-r+f\cdot\ell+2t\\
\ch_1&=-f+f\cdot(\ell+2h)\hat h+(f\cdot h-t)\hat\ell\\
\ch_2&=-2f\cdot\ell-5t
\end{align*}
where $\ch(F)=(r,f,t)$. Moreover, $\Phi(\OO)=\OO[-1]$.
\end{thm}
The last part of the theorem follows easily
from \propref{p:necconds} and by applying
$\R q_*$ to the defining extension of $\E$.
\begin{remark}
One can also prove that the collection $\{E_x\}$ of simple sheaves is actually
a moduli space of $\mu$-stable bundles with Mukai vector $(2,\ell,-3)$
(see \cite{BBH}).
Moreover, this is naturally polarized by $\hat H$.
The set of restrictions to points in the first factor $\{{}_yE\}$
is also a collection of $\mu$-stable vector bundles
with respect to $\hat H$ with Mukai vector $(2,-\hat\ell,-3)$.
\end{remark}

\section{Degenerate reflexive K3 surfaces}
We now consider the case when $X$ is a degenerate K3 surface. We shall first
state a technical result concerning the geometry of the effective divisor
$\ell+2h$ the proof of which can be found in the appendix.
\begin{lemma}
The divisor $\ell+2h$ is a sum of rational curves. Moreover, there are two
effective divisors $d_1$ and $d_2$ such that $\ell+2h=d_1+d_2$ with
$d_1\cdot d_2=0$, $d_1^2=-2=d_2^2$ and neither $d_1-d_2$ nor
$d_2-d_1$ are effective.
\end{lemma}
Let $L$ and $H$ denote the line bundles $\OO(\ell)$ and $\OO(h)$,
respectively. The degrees of divisors will be taken with respect to $h$.
Note that $(\ell+2h)^2=-4$. So $\chi(LH^2)=0$. We also have $\deg(\ell+2h)=4$.
\begin{remark}
If $\deg(d_1)=1$ then $\deg(d_2)=3$ and we must have $h=d_1+e$, where $e$ is a
rational curve of degree 1 with $d_1\cdot e=3$. This is because
$\chi(D_1H^*)=1$ and so $h-d_1$ is effective. Since $(d_1+e)\cdot d_1=1$ we
have $d_1\cdot e=3$ and $(d_1+e)^2=2$ implies that $e^2=-2$ and $h\cdot e=1$.
We shall call such a degenerate reflexive K3 surface a {\it type II\/} surface
and the case when $\deg(d_1)=2$, a {\it type I\/} surface.
\end{remark}
We can now use this lemma to establish the existence of Fourier-Mukai
transforms for both types of degenerate reflexive K3 surfaces.
\begin{thm}\label{t:degFM}
Suppose that $X$ is a type I K3 surface.
Write $\ell+2h=d_1+d_2$ using \lemref{l:d1d2}.
Then there is a moduli space of simple bundles with Mukai vector 
$(2,\ell,-3)$ which is isomorphic to $X$ with deformation sheaf
$\E$ over $X\cross X$ given by the extension
$$0\lra D_1H^*\boxtimes D^*_1H\lra\E\lra(D_2H^*\boxtimes
D^*_2H)\I_\Delta\lra0.$$
This sheaf gives rise to a Fourier-Mukai transform 
normalized by $\Phi(\OO)=\OO[-1]$. 
The Chern character of $\Phi(F)$ is given by
\begin{align*}
\ch_0&=-r+f\cdot\ell+2t\\
\ch_1&=-f-t\ell-(f\cdot h)(\ell+2h)+(f\cdot\ell)\ell
-(f\cdot d_1)d_1-(f\cdot d_2)d_2\\
\ch_2&=-2f\cdot\ell-5t,
\end{align*}
where $\ch(F)=(r,f,t)$.
If $X$ is type II surface then $\E$ exists as above but is given by the
extension 
$$0\lra D_1H^*\boxtimes D_2D_1^{-2}H\lra\E\lra(D_2H^*\boxtimes
D_1^*H)\I_\Delta\lra0,$$
where $\deg(D_1)=1$.
This also gives rise to Fourier-Mukai transform and the Chern character of the
transform of a sheaf $F$ is
\begin{align*}
\ch_0&=-r+f\cdot\ell+2t\\
\ch_1&=-f+(f\cdot\ell)h+(f\cdot d_1)d_2-(f\cdot(2d_1+d_2))d_1+t(d_2-3d_1+2h)\\
\ch_2&=-2f\cdot\ell-5t.
\end{align*}
The transform is normalized by $\det\Phi(\OO)=\OO$ but $\OO$ does not
satisfy $\Phi$-WIT.
\end{thm}
\noindent Recall that a sheaf $F$ satisfies $\Phi$-WIT$_i$ if for all $j\neq i$,
$R^j\Phi(F)=0$ and satisfies $\Phi$-WIT if it satisfies $\Phi$-WIT$_i$ for any
$i$.
\begin{proof}
We obtain the existence of the Fourier-Mukai transforms in both cases by
observing that the values of $A$, $B$, $C$ and $D$ satisfy the sufficient
conditions of \thmref{t:suffic}.
The formulae for the Chern character of $\Phi(F)$
comes from the general formula in \propref{p:necconds}.

To compute $\Phi(\OO)$ we apply $\R q_*$ to
the extensions. For $\deg(d_1)=2$ we have $H^i(D_1H^*)=0=H^i(D_2H^*)$
for all $i$ and so $\R q_*((D_2H^*\boxtimes
D^*_2H)\I_\Delta)=\OO[-1]$ and $\R\hat\pi_*(D_1H^*\boxtimes D_1^*H)=0$.
This implies that $\Phi(\OO)=\OO[-1]$.

If $\deg(d_1)=1$ then $\R q_*(D_1H^*\boxtimes
D_2D_1^{-2}H)=D_2D_1^{-2}H[-2]$ and $\R q_*((D_2H^*\boxtimes
D^*_1H)\I_\Delta)$ is concentrated in position 1 and is an extension of
$D_2D_1^*$ by $D_1^*H$, which we shall denote by $R^1$. Then we have an exact
sequence
$$0\lra R^1\Phi(\OO)\lra R^1\lra D_2D_1^{-2}H\lra R^2\Phi(\OO)\lra0.$$
Hence, $\det\Phi(\OO)=\OO$. But, since $H^0(D_1^{-3}H)=0$ as $\deg(d_1)=1$ we
have that $D_2D_1^*\to R^1$ lifts to $R^1\Phi(\OO)$ and so there is a map
$D_1^*H\to D_2D_1^{-2}H$. But this comes from a section of $D_2D_1^*$ which
must be zero and so $\rk(R^1\Phi\OO)=2$ and hence $R^1\Phi(\OO)=R^1$ and
$R^2\Phi(\OO)=D_2D_1^{-2}H$. 
\end{proof}

\section{Applications to the biholomorphic classification of moduli spaces}

One of the main uses of Fourier-Mukai transforms is to give biholomorphic
isomorphisms between components of the moduli of simple sheaves. 
Recall from \cite{Muk1} that Fourier-Mukai transforms satisfy an analogue of
the Parseval theorem which states that if a
sheaf $F$ satisfies $\Phi$-WIT with transform $\hat F$ then
$\Ext^i(F,F)=\Ext^i(\hat F,\hat F)$ 
and so $\hat F$ is simple if and only if $F$ is simple.
Moreover, since the Zariski tangent space at $[F]$ to the moduli scheme of
simple sheaves is given by $\Ext^1(F,F)$ then if all the geometric points of a
component satisfy $\Phi$-WIT, $\Phi$ gives an isomorphism of components.
We have seen an example of this in the main theorem of \cite{BM}. The
following theorems can all be proved in the same way as the main theorem of
\cite{BM}. 
\begin{thm}
If $X$ is reflexive then, from \thmref{t:nondegFM} or \thmref{t:degFM},
$\Phi$ gives a map $\Hilb^nX$ to the moduli of simple sheaves with Mukai vector
$(1+2n,\pm n\ell,1-3n)$ given by $\I_W\to R^1\Phi\I_W$. By the above argument
this must be an isomorphism onto a component of the simple moduli scheme of
bundles.
\end{thm}
To prove this we need only apply $\Phi$ to the structure sequence
of $\I_W$ to obtain $0\to R^0\Phi\OO_W\to R^1\Phi\I_W\to\OO\to0$. In fact the
image is a moduli scheme of Gieseker stable sheaves when $X$ is
non-degenerate. 
\begin{thm}
Consider a K3 surface with a line bundle $M$ such that
$H^*(M)=0$. Then there is a Fourier-Mukai transform which gives an isomorphism
between $\Hilb^n X$ and a component of the moduli scheme of simple sheaves
with Mukai vector $(2n-1,\pm nm,-n-1)$.
\end{thm}

In this case, \thmref{t:Lnocohom} implies that we have a Fourier-Mukai
transform $\Phi$ with $\Phi(\OO)=\OO$.
Then if we apply $\Phi$ to the
structure sequence of $\I_W$ for $W\in\Hilb^n X$, we obtain the long exact
sequence 
$$0\to R^0\Phi\I_W\to\OO\to R^0\Phi\OO_W\to R^1\Phi\I_W\to0.$$
But the middle map cannot vanish as $\Phi$ is a natural isomorphism. This
means that $R^0\Phi\I_W=0$ and so, again, we obtain an isomorphism
between $\Hilb^n X$ and a component of the moduli scheme of sheaves this time
with Mukai vector $(2n-1,\pm nm,-n-1)$. This also applies to the case of a
reflexive K3 surface
where we take $m=l+2h$ when $X$ is non-degenerate, or $m=d_1-d_2$ when $X$ is
degenerate. 
\section*{Appendix}
We establish the following technical lemma on the existence
of reductions of $\ell+2h$ on a reflexive K3 surface.
\begin{lemma}\label{l:d1d2}
The divisor $\ell+2h$ is a sum of rational curves. Moreover, there are two
effective divisors $d_1$ and $d_2$ such that $\ell+2h=d_1+d_2$ with
$d_1\cdot d_2=0$, $d_1^2=-2=d_2^2$ and neither $d_1-d_2$ nor
$d_2-d_1$ are effective.
\end{lemma}
Let $L$ and $H$ denote the line bundles $\OO(\ell)$ and $\OO(h)$,
respectively. The degrees of divisors will be taken with respect to $h$.
Note that $(\ell+2h)^2=-4$. So $\chi(LH^2)=0$. We also have $\deg(\ell+2h)=4$.
\begin{proof} 
Suppose that $\ell+2h$ is effective. Then, by the adjunction formula,
$\ell+2h$ cannot be irreducible. We can also see that $\ell+2h$ is not a
multiple of an irreducible divisor $d_1$ as $d_1=-2$ in this case. So
$\deg(d_1)\leq3$. Let $d_1$
be an irreducible component of $\ell+2h$. By the Hodge Index Theorem
we have $d_1^2\leq4$. But if
$d_1^2=4$ then $\deg(d_1)=3$ and so $d_2=\ell+2h-d_1$ is irreducible and
$d_1\cdot d_2>0$ so we cannot have $(d_1+d_2)^2=-4$. Now suppose that
$d_1^2=2$. Then if $\deg(d_1)=3$ we have $d_2$ irreducible again and this is
also a contradiction. So we have $\deg(d_1)=2$. But then $(d_1-h)^2=0$ and
$\deg(d_1-h)=0$ so the Hodge Index Theorem
implies that $d_1=h$. But then $\ell+h$ would be
effective and this is impossible because $(\ell+h)^2=-10$ whereas the degree is
2 so it would have to be the sum of two nodal curves with intersection $-3$; a
contradiction. 

The preceding argument has shown that the irreducible components of $\ell+2h$
are either elliptic curves or rational curves which we can assume are smooth
by Bertini. We shall show that the elliptic
curve case is also impossible. Suppose that $d_1$ is an elliptic curve.
Suppose that
$\deg(d_1)=1$. Then the linear system of $d_1$ expresses $X$ as an elliptic
fibration over $h$. This is not possible as $X$ can only be an elliptic
fibration over $\PP^1$. So we have $\deg(d_1)=2$. Let $d_2=\ell+2h-d_1$.
Since $d_1\cdot d_2\geq0$ we see that $d_2^2\leq-4$ and so $d_2$ cannot be
irreducible. Let $d_2=d_3+d_4$ with $d_3$ and $d_4$ effective. Since the
degrees of $d_3$ and $d_4$ must be 1, they must both be rational. Suppose now
that $d_3\neq d_4$. Then $d_3\cdot d_4\geq0$ and hence $d_1\cdot d_3=d_1\cdot
d_4=0$. Then we can use $d_3$ and $d_1+d_4$ since $d_3^2=-2$,
$(d_1+d_4)^2=-2$, $d_3\cdot(d_1+d_4)=0$ and neither $d_1+d_4-d_3$ nor its
negation are effective.
We are now left with the case where
$d_3=d_4$. Then, from $(d_1+2d_3)^2=-4$ we obtain $d_1\cdot d_3=1$. This
implies that $X$ is fibred over $d_3$. But this fibration must be
locally-trivial and since $d_1\subset X$ is a section, it must be 
trivial. This is not possible and so we have established the first part of the
lemma.

We now write
$$\ell+2h=\sum_{i=1}^nc_i,$$
where $n=2,3,4$ and $c_i$ are rational curves. If $n=2$ then we must have
$c_1\cdot c_2=0$ as
\begin{equation}
(\ell+2h)^2=-4.\label{e:squr}
\end{equation}
So we can set $d_1=c_1$ and $d_2=c_2$.
If $n=3$ then the condition (\ref{e:squr}) can be written
$$c_1\cdot c_2+c_2\cdot c_3+c_3\cdot c_1=1.$$
If $c_1=c_2$ then this reads $c_1\cdot c_3=3/2$, which is impossible. This
implies that the $c_i$ are distinct and so $c_i\cdot c_j\geq0$ for $i\neq j$.
Then (\ref{e:squr}) implies that only one of $c_i\cdot c_j$ is non-zero for
$i\neq j$. Without loss of generality, we suppose $c_1\cdot c_2=1$. Then set
$d_1=c_3$ and $d_2=c_1+c_2$. These have the desired properties.

Finally, suppose that $n=4$. The condition (\ref{e:squr}) reads
$$\sum_{i<j}c_i\cdot c_j=2.$$
If $c_1=c_2=c_3$ then we have $3c_1\cdot c_4=8$ which is a contradiction. 
Note that $(c_i+c_j)^2\leq-2$ for $i\neq j$ since, otherwise, the linear system
$|c_i+c_j|$ would be non-trivial and the resulting deformation would be a
non-rational irreducible curve contradicting the lemma. In particular, we must
have $c_i\cdot c_j\leq1$ for $i\neq j$. This shows that, if
$c_1=c_2$ then (\ref{e:squr}), which reads $2c_1\cdot c_2+2c_1\cdot
c_4+c_3\cdot c_4=4$, implies that $c_3\cdot c_4=0$ and $c_1\cdot
c_3=1=c_1\cdot c_4$. So we set $d_1=c_1+c_3$ and $d_2=c_1+c_4$. These
satisfy the conditions of the proposition.
On the other hand, if the $c_i$ are pairwise distinct the $0\leq c_i\cdot
c_j\leq1$ and we have two possibilities (up to permutation): either (a)
$c_1\cdot c_2=1=c_3\cdot c_4$ or (b) $c_1\cdot c_i=0$ for $i\neq1$. In case
(a), we set $d_1=c_1+c_2$ and $d_2=c_3+c_4$. In case (b), we set $d_1=c_1$ and
$d_2=c_2+c_3+c_4$.
\end{proof}

\bibliographystyle{plain}

\end{document}